\documentclass[11pt]{amsart}
\usepackage[margin=1.5in]{geometry} 
\usepackage{amsmath,amssymb}
\usepackage{enumerate}
\usepackage{color}
\usepackage{graphicx}
\usepackage{dsfont}

\newtheorem{theorem}{Theorem}[section]
\newtheorem{proposition}[theorem]{Proposition}
\newtheorem{lemma}[theorem]{Lemma}
\newtheorem{corollary}[theorem]{Corollary}
\theoremstyle{definition}

\newtheorem{remark}[theorem]{Remark}
\newtheorem{example}[theorem]{Example}

\newcommand{\R}{\mathbb{R}}

\newcommand{\N}{\mathbb{N}}

\renewcommand{\epsilon}{\varepsilon}

\newcommand{\X}{\mathcal{X}}
\newcommand{\Y}{\mathcal{Y}}
\newcommand{\eps}{\varepsilon}
\DeclareMathOperator{\spt}{spt}
\newcommand{\1}{\mathbf{1}}
\DeclareMathOperator*{\argmin}{arg\, min}
\DeclareMathOperator*{\essinf}{ess\, inf}
\DeclareMathOperator*{\esssup}{ess\, sup}

\numberwithin{equation}{section}
\usepackage[pdfborder={0 0 0}]{hyperref}
\hypersetup{
  urlcolor = black,
  pdfauthor = {Marcel Nutz, Johannes Wiesel},
  pdfkeywords = {Optimal Transport; Entropic Regularization; Schrodinger potentials; Kantorovich potentials},
  pdftitle = {Entropic Optimal Transport: Convergence of Potentials},
  pdfsubject = {Entropic Optimal Transport: Convergence of Potentials},
  pdfpagemode = UseNone
}

\begin{document}

\author{Marcel Nutz}
\thanks{The authors thank Guillaume Carlier, Giovanni Conforti, Soumik Pal and Luca Tamanini for helpful discussions.}
\thanks{MN acknowledges support by an Alfred P.\ Sloan Fellowship and NSF Grants DMS-1812661,  DMS-2106056.}
\address[MN]{Departments of Statistics and Mathematics, Columbia University,  1255 Amsterdam Avenue, New York, NY 10027, USA}
\email{mnutz@columbia.edu}

\author{Johannes Wiesel}
\address[JW]{Department of Statistics, Columbia University, 1255 Amsterdam Avenue,
New York, NY 10027, USA}
\email{johannes.wiesel@columbia.edu}

\title[Entropic Optimal Transport]{Entropic Optimal Transport:\\Convergence of Potentials}

\date{\today}

\begin{abstract}
  We study the potential functions that determine the optimal density for $\varepsilon$-entropically  regularized optimal transport, the so-called Schr\"odinger potentials, and their convergence to the counterparts in classical optimal transport, the Kantorovich potentials. In the limit $\varepsilon\to0$ of vanishing regularization, strong compactness holds in $L^{1}$ and cluster points are Kantorovich potentials. In particular, the Schr\"odinger potentials converge in $L^{1}$ to the Kantorovich potentials as soon as the latter are unique. These results are proved for all continuous, integrable cost functions on Polish spaces. In the language of Schr\"odinger bridges, the limit corresponds to the small-noise regime. %
\end{abstract}

\keywords{Optimal Transport; Entropic Regularization; Schr\"odinger potentials}
\subjclass[2010]{90C25; 49N05}

\maketitle

\section{Introduction and Main Result}

Let $(\X,\mu)$ and $(\Y,\nu)$ be Polish probability spaces and $\Pi(\mu,\nu)$ the set of all couplings; i.e., probability measures~$\pi$ on~$\X\times\Y$ with first marginal~$\mu$ and second marginal~$\nu$. Moreover, let $c:\X\times\Y\to\R_{+}$ be continuous with 
\begin{equation}\label{eq:cIntegrable}
 \int c(x,y)\,\mu(dx)\nu(dy)<\infty.
\end{equation}
Given a constant~$\eps>0$, the entropic optimal transport (EOT) problem is
\begin{equation}\label{eq:EOT}
  I_{\eps} :=\inf_{\pi\in\Pi(\mu,\nu)} \int_{\X\times\Y} c(x,y) \, \pi(dx,dy) + \eps H(\pi|\mu\otimes\nu), 
\end{equation}
where $H(\,\cdot\,|\mu\otimes\nu)$ denotes relative entropy with respect to the product measure,
$$
  H(\pi|\mu\otimes\nu):=\begin{cases}
\int \log (\frac{d\pi}{d(\mu\otimes\nu)}) \,d\pi, & \pi\ll \mu\otimes\nu,\\
\infty, & \pi\not\ll \mu\otimes\nu.
\end{cases} 
$$
For $\eps=0$ we recover the Monge--Kantorovich optimal transport problem, and~\eqref{eq:EOT} can be seen as its entropic regularization with  parameter~$\eps>0$. The minimization~\eqref{eq:EOT} admits a unique solution $\pi_{\eps}\in\Pi(\mu,\nu)$; moreover, $\pi_{\eps}\sim \mu\otimes\nu$ and its density is of the form
\begin{equation}\label{eq:densityFormIntro}
\frac{d\pi_{\eps}}{d(\mu\otimes\nu)}(x,y) = \exp \left(\frac{f_{\eps}(x) +g_{\eps}(y) - c(x,y)}{\eps}\right)
\end{equation}
for two measurable functions $f_{\eps}: \X\to\R$  and $g_{\eps}: \Y\to\R$. We call these functions the \emph{Schr\"odinger potentials}. They are unique up to normalization: any constant can be added to $f_{\eps}$ and subtracted from $g_{\eps}$. The integrability~\eqref{eq:cIntegrable} of~$c$ implies that $f_{\eps}\in L^{1}(\mu)$ and $g_{\eps}\in L^{1}(\nu)$, and we enforce the symmetric normalization
\begin{equation}\label{eq:normalizationIntro}
\int f_\epsilon(x)\,\mu(dx)=\int g_\epsilon(y)\,\nu(dy)
\end{equation}
to have uniqueness of the potentials in all that follows. We mention that~$\pi_{\eps}$ can be characterized as the unique coupling $\pi\in\Pi(\mu,\nu)$ whose density is of the form~\eqref{eq:densityFormIntro}. See, for instance,
~\cite[Statements~3.6, 3.15, 3.19, 3.38]{FollmerGantert.97} for existence and uniqueness,  
or~\cite{Nutz.20} for a simple derivation including integrability under~\eqref{eq:cIntegrable}. 
These result heavily build on~\cite{BorweinLewis.92, BorweinLewisNussbaum.94, Csiszar.75, RuschendorfThomsen.97}, among others.
Rewriting the minimization~\eqref{eq:EOT}, the coupling $\pi_{\eps}$ can be interpreted as the so-called static Schr\"odinger bridge 
\begin{equation}\label{eq:bridgeIntro}
   \pi_{\eps}=\argmin_{\pi\in \Pi(\mu,\nu)} H(\pi|R)
\end{equation}
for the reference probability $dR\propto e^{-c/\eps}d(\mu\otimes\nu)$ which elucidates~\eqref{eq:densityFormIntro} as the factorization property~$\frac{d\pi_{\eps}}{dR}(x,y)= e^{f_{\eps}(x)/\eps}e^{g_{\eps}(y)/\eps}=:F(x)G(y)$.\footnote{We mention that \cite{Leonard.14} uses the term Schr\"odinger potentials for $f_{\eps}/\eps,g_{\eps}/\eps$  in the Schr\"odinger bridge context, as is natural when no parameter~$\eps$ is present. On the other hand, calling $f_{\eps},g_{\eps}$ potentials is more convenient in our setting, well motivated by the connection with Kantorovich potentials in Theorem~\ref{thm:1}, and consistent with the terminology in~\cite{GigliTamanini.21}.} 
A closely related, more analytic way to characterize the potentials are the Schr\"odinger equations. Writing also $C(x,y)=e^{-c(x,y)/\eps}$, the fact that $\pi_{\eps}$ of~\eqref{eq:densityFormIntro} is in $\Pi(\mu,\nu)$ implies that $(F,G)$ solves the coupled equations
\begin{equation}\label{eq:SchrodingerEqnsIntro}
  F(x)^{-1}= \int G(y) C(x,y)\, \nu(dy) \quad \mu\mbox{-a.s.}, \quad G(x)^{-1}= \int F(y) C(x,y)\, \nu(dx) \quad \nu\mbox{-a.s.}
\end{equation}
Conversely, we can use any solution $(F,G)$ to define %
a coupling with density of the form~\eqref{eq:normalizationIntro}. This coupling must coincide with $\pi_{\eps}$ by the aforementioned uniqueness, and then $(\eps\log F,\eps\log G)$ must be our Schr\"odinger potentials $(f_{\eps},g_{\eps})$, up to normalization. We refer to~\cite{Beurling.60,RuschendorfThomsen.93} and the references therein for more on Schr\"odinger equations, and to~\cite{Follmer.88,Leonard.14} for extensive surveys on Schr\"odinger bridges.

Yet another way to introduce the potentials is to consider the dual problem of~\eqref{eq:EOT} in the sense of convex analysis,
\begin{align}
\begin{split}\label{eq:dualEOT}
S_\epsilon
:=\sup_{f\in L^1(\mu), g \in L^1(\nu)} \bigg(\int f(x)\, \mu(dx)&+\int g(y) \,\nu(dy) \\
&-\epsilon \int e^{\frac{f(x)+g(y)-c(x,y)}{\epsilon}}\, \mu(dx)\nu(dy)+\epsilon\bigg).
\end{split}
\end{align}
Then $(f_{\eps},g_{\eps})$ is the unique solution of~\eqref{eq:dualEOT} with the normalization~\eqref{eq:normalizationIntro}. Indeed, direct arguments show the weak duality $S_{\eps}\leq I_{\eps}$. To see that equality is attained by~$(f_{\eps},g_{\eps})$ and~$\pi_{\eps}$, we plug in~\eqref{eq:densityFormIntro} and use $\pi_{\eps}(\X\times\Y)=1$ to find that $S_{\eps}\geq \int f_\epsilon(x)\,\mu(dx)+\int g_\epsilon(y)\,\nu(dy)\geq I_{\eps}$. Uniqueness holds by strict concavity. See \cite{PennanenPerkkio.19} and the references therein for a convex analysis perspective including~\eqref{eq:dualEOT}.

We are interested in the relation of $(f_{\eps}, g_{\eps})$ to solutions of the dual Monge--Kantorovich problem,
\begin{align}
\label{eq:dualOT}
S_0 :=\sup_{f\in L^1(\mu),\, g \in L^1(\nu),\, f\oplus g\leq c} \bigg(\int f(x)\, \mu(dx) + \int g(y) \,\nu(dy) \bigg),
\end{align}
where $(f\oplus g)(x,y):= f(x)+g(y)$.
It is well known that $S_{0}=I_{0}$ and that a solution $(f_{0},g_{0})$ exists \cite[Theorem~5.10, Remark~5.14]{Villani.09}. (In fact, Theorem~\ref{thm:1}
 below yields another proof as a by-product.) There is the same ambiguity as above, and to streamline terminology, we call $(f_{0},g_{0})$ Kantorovich potentials if they satisfy the normalization~\eqref{eq:normalizationIntro} for $\eps=0$.
As~\eqref{eq:dualOT} lacks the strict convexity of~\eqref{eq:dualEOT}, multiple Kantorovich potentials may exist even after normalization, for instance when both marginals are discrete. 
Nevertheless, uniqueness of Kantorovich potentials is known to hold for most problems of  interest to us, especially when~$c$ is differentiable and at least one marginal support is connected. See for instance \cite[Appendix~B]{BerntonGhosalNutz.21} for sufficient conditions.

Much of the enormous recent interest in entropic optimal transport stems from the success of Sinkhorn's algorithm in high-dimensional problems, enabling data-rich applications in areas like machine learning or image processing. Popularized in this context by~\cite{Cuturi.13}, Sinkhorn's algorithm computes the Schr\"odiger potentials~$(f_{\eps},g_{\eps})$ by alternating projections. From a computational point of view, the Monge--Kantorovich problem is significantly harder than the entropic one; see \cite{PeyreCuturi.19} for a recent survey and numerous references. It is therefore natural to investigate $(f_{\eps},g_{\eps})$ as~$\eps\to0$ to approximate Kantorovich potentials.

On the primal side, weak compactness of $\Pi(\mu,\nu)$ immediately implies that $(\pi_{\eps})$ admits cluster points as~$\eps\to0$. Moreover, any cluster point is an optimal transport, so that if uniqueness is known for the solution~$\pi_{0}$ of the limiting optimal transport problem, then $\pi_{\eps}\to\pi_{0}$. See \cite{CarlierDuvalPeyreSchmitzer.17,Leonard.12} for proofs by Gamma convergence, or~\cite{BerntonGhosalNutz.21} for a geometric proof assuming only continuity of~$c$. Our aim is to establish a comparable result on the dual side. Here, compactness is not obvious (unless $\mu,\nu$ are compactly supported). Our main result provides strong compactness in~$L^{1}$ for $(f_{\eps})$ and $(g_{\eps})$ as $\eps\to0$, and moreover, that cluster points are Kantorovich potentials. In most cases of interest, the latter are unique, so that the whole sequence converges.

\begin{theorem}\label{thm:1}
Let $(f_{\eps},g_{\eps})$ be the unique Schr\"odinger potentials for $\eps>0$.
\begin{itemize}
\item[(a)] Given $\eps_{n}\to0$, there is a subsequence $(\eps_{k})$ such that $f_{\eps_{k}}$ converges in $L^{1}(\mu)$ and $g_{\eps_{k}}$ converges in $L^{1}(\nu)$.
\item[(b)] If $\lim_{n} f_{\eps_{n}} = f$ $\mu$-a.s.\ and $\lim_{n} g_{\eps_{n}} = g$ $\nu$-a.s.\ for $\eps_{n}\to0$, then $(f,g)$ are Kantorovich potentials and the convergence also holds in $L^{1}$.
\end{itemize}
If the Kantorovich potentials $(f_{0},g_{0})$ for~\eqref{eq:dualOT} are unique, it follows that $\lim_{\eps} f_{\eps} = f_{0}$ in $L^{1}(\mu)$ and  $\lim_{\eps} g_{\eps} = g_{0}$ in $L^{1}(\nu)$.
\end{theorem}

Applications of interest for Theorem~\ref{thm:1} include costs $c(x,y)=\|x-y\|^{2}$ on~$\X=\Y=\R^{d}$ with unbounded marginal supports as in~\cite{MenaWeed.19}; here~$c$ is continuous but not uniformly continuous.
Theorem~\ref{thm:1} simplifies substantially in the case of compactly supported marginals. More generally, if~$c$ is uniformly continuous, the functions $f_{\eps},g_{\eps}$ inherit its modulus of continuity (uniformly in~$\eps$) and then uniform convergence on compact subsets along a subsequence follows from the Arzel\`a--Ascoli theorem; cf.\ Proposition~\ref{pr:unifContCase}. A result along those lines is contained in~\cite[Section~5]{GigliTamanini.21} in the particular case of quadratic cost and compact marginals. We emphasize that~\cite{GigliTamanini.21} analyzes the more complex dynamic problem of approximating $W_{2}$ geodesics with entropic interpolation; the present static setting would correspond only to its marginals at times~$t=0,1$. 
Given the results of~\cite{GigliTamanini.21}, one may conjecture that Theorem~\ref{thm:1} can be extended to interpolations and intermediate times~$t\in(0,1).$

When $\X,\Y$ are finite sets, optimal transport is a finite-dimensional linear programming problem. For such problems, a detailed convergence analysis of entropic regularization is presented in~\cite{CominettiSanMartin.94}. In particular, convergence holds even when Kantorovich potentials are not unique.

Theorem~\ref{thm:1} can be related to the large deviations principle (LDP) of~\cite{BerntonGhosalNutz.21} which describes the convergence of $(\pi_{\eps})$ on the primal side (cf.\ Section~\ref{se:LDP} for a detailed discussion). On compact spaces, convergence of potentials is equivalent to the validity of an LDP whose rate functions includes the limiting Kantorovich potentials. On the other hand, neither result implies the other in general, and we see the results and methods as complementary. Indeed, the ``easier'' inequality for the present dual approach corresponds to the more delicate one in the primal approach, and vice versa. See also \cite{ConfortiTamanini.19,Pal.19} for expansions of the entropic transport cost as $\eps\to0$, which are related to the speed of convergence of $(\pi_{\eps})$. Finally,  we mention~\cite{Berman.20}, studying the convergence of the discrete Sinkhorn algorithm to an optimal transport potential in the joint limit when $\eps_{n}\to0$ and the marginals $\mu,\nu$ are approximated by discretizations $\mu_{n},\nu_{n}$ satisfying a certain density property.
Beyond the aforementioned special cases and connections, Theorem~\ref{thm:1} is novel, to the best of our knowledge.

Two extensions of Theorem~\ref{thm:1} are obtained in the body of the text. The first one replaces~$c$ in~\eqref{eq:EOT} by a cost function $c_{\eps}$ that may depend on~$\eps$ and converges to the continuous cost~$c$ of the Monge--Kantorovich problem as~$\eps\to0$. This extension demonstrates the stability of the convergence in Theorem~\ref{thm:1}. In addition, it may be a natural result from the perspective of Sch\"odinger bridges (see~\cite{Leonard.14}): the corresponding reference measures $R_{\eps}$ in~\eqref{eq:bridgeIntro} are those with large deviations rate~$c$. The second extension replaces the two marginals $(\mu,\nu)$ by any (finite) number of marginals. The resulting ``multimarginal'' optimal transport problem has become a focus of attention as the primary tool to analyze Wasserstein barycenters in the sense of~\cite{AguehCarlier.11}. 
Its entropic regularization again admits a version of Sinkhorn's algorithm; see~\cite{Carlier.21} for a very recent analysis showing linear convergence and further references.
The techniques developed in the proof of Theorem~\ref{thm:1} are quite versatile and extend to the multimarginal setting without effort.

The remainder of this paper is organized as follows. Section~\ref{sec:auxiliaryResults} collects auxiliary results for the proof of Theorem~\ref{thm:1}, which is carried out in Section~\ref{se:ProofOfMainRes} and followed by the specialization to uniformly continuous costs. The relation with the LDP is the subject of Section~\ref{se:LDP}. In Section~\ref{se:varyingCost} we present the extension to costs~$c_{\eps}$ that vary with~$\eps$, and Section~\ref{sec:multi} concludes with the multimarginal case.

\section{Auxiliary Results}\label{sec:auxiliaryResults}

In this section we collect a number of auxiliary results for the proof of Theorem~\ref{thm:1}. Anticipating the generalization in Section~\ref{se:varyingCost}, we remark that the statements and proofs in this section hold for any measurable (but not necessarily continuous) cost function $c:\X\times\Y\to\R_{+}$ that is integrable in the sense of~\eqref{eq:cIntegrable}; further regularity is only required in Lemma~\ref{lem:lusin}, where the condition is stated explicitly.

Let $\eps>0$. We recall the Schr\"odinger potentials $f_{\eps}\in L^{1}(\mu)$ and $g_{\eps}\in L^{1}(\nu)$ from the Introduction and in particular the  normalization 
\begin{equation}\label{eq:normalization}
\int f_\epsilon(x)\,\mu(dx)=\int g_\epsilon(y)\,\nu(dy)=S_\epsilon/2\ge 0.
\end{equation}
The fact that $\pi_{\eps}$ of~\eqref{eq:densityFormIntro} is a probability measure with  marginals~$\mu$ and~$\nu$ implies
\begin{align}\label{eq:one}
\int  e^{\frac{f_\epsilon(x)+g_\epsilon(y)-c(x,y)}{\epsilon}}\, \nu(dy)=1\quad \mu\mbox{-a.s.}, \quad \int  e^{\frac{f_\epsilon(x)+g_\epsilon(y)-c(x,y)}{\epsilon}}\, \mu(dx)=1 \quad \nu\mbox{-a.s.}
\end{align}
and hence the Schr\"odinger equations
\begin{align}
\begin{split}\label{eq:definition}
f_\epsilon(x)&= -\epsilon\log \int e^{\frac{g_\epsilon(y)-c(x,y)}{\epsilon}}\, \nu(dy) \quad \mu\mbox{-a.s.},\\
g_\epsilon(y) &= -\epsilon\log \int e^{\frac{f_\epsilon(x)-c(x,y)}{\epsilon}} \, \mu(dx) \quad \nu\mbox{-a.s.}
\end{split}
\end{align}
By choosing versions of $f_\eps,g_\eps$ we may and will assume that these conjugacy relations hold everywhere on~$\X\times\Y$. In particular, this provides canonical extensions of $f_\eps,g_\eps$ to the whole marginal space. The conjugacy relations can also be used to obtain a priori estimates, as has been previously exploited in~\cite{Carlier.21,DiMarinoGerolin.20}, among others.

\begin{lemma}\label{lem:1}
For all $x\in \X$ and $y \in \Y$, we have
\begin{align*}
\inf_{y\in \Y} \big[c(x,y)-g_\epsilon(y)\big] \le f_\epsilon(x)\le \int c(x,y)\, \nu(dy),\\
\inf_{x\in \X} \big[c(x,y)-f_\epsilon(x)\big]\le g_\epsilon(y)\le \int c(x,y)\, \mu(dx).
\end{align*}
\end{lemma}

\begin{proof}
Using~\eqref{eq:definition}, Jensen's inequality and~\eqref{eq:normalization}, %
\begin{align}\label{eq:bounded_from_above}
\begin{split}
f_\epsilon(x)&=-\epsilon \log \int e^{\frac{g_\epsilon(y)-c(x,y)}{\epsilon}}\,\nu(dy)\\
&\le \int \left[ -g_\epsilon(y)+c(x,y)\right]\, \nu(dy) \le \int c(x,y)\, \nu(dy),
\end{split}
\end{align}
which is the upper bound. For the lower bound we note that by~\eqref{eq:definition},
\begin{align*}
f_\epsilon(x)%
&\ge -\epsilon \log \int e^{\frac{\sup_{y\in \Y} [g_\epsilon(y)-c(x,y)]}{\epsilon}} \, \nu(dy)\\
&= -\sup_{y\in \Y} \big[g_\epsilon(y)-c(x,y)\big]=\inf_{y\in \Y} \big[c(x,y)-g_\epsilon(y)\big].
\end{align*}
The proof for $g_\epsilon$ is symmetric.
\end{proof}

Let $(M,d)$ be a metric space. A function $\omega:\R_{+}\to\R_{+}$ is a modulus of continuity if it is continuous at~$0$ with $\omega(0)=0$. More generally, we call $\omega:M\times \R_{+}\to\R_{+}$ a modulus of continuity if $\omega(x,\cdot)$ has those properties for each $x\in M$. A function $F: M\to\R$ is $\omega$-continuous if it admits the modulus of continuity $\omega(x,\cdot)$ at $x\in M$; that is, $|F(x)-F(x')|\le \omega(x,d(x,x'))$ for all $x,x'\in M$. To avoid ambiguity, we say that $F$ is uniformly $\omega$-continuous if $\omega$ can be chosen independent of~$x$. The following generalization of the Arzel\`a--Ascoli theorem will be used to construct  limits of $f_{\eps}$ and $g_{\eps}$.

\begin{lemma}\label{lem:aa1}
Let $(M,d)$ be a separable metric space and let $(F_n)$ be (arbitrary) functions on $M$ which are pointwise bounded and satisfy
\begin{align}\label{eq:equicontinuity}
|F_n(x_1)-F_n(x_2)|\le \omega(x_1,d(x_1,x_2))+h_n, \quad x_1, x_2\in M
\end{align}
for some modulus of continuity~$\omega:M\times \R_{+}\to\R_{+}$ and a sequence $h_{n}\to0$ of constants. Then after passing to a subsequence, $(F_n)$ converges uniformly on compact subsets to a $\omega$-continuous function $F: M\to\R$.
\end{lemma}

\begin{proof}
Let $D\subset M$ be a countable dense set, fix $\delta>0$ and choose $n_0\in \N$ such that $h_n\le \delta/6$ for all $n\ge n_0$. As $(F_n)$ is pointwise bounded, a diagonal argument yields a subsequence, still denoted $(F_n)$, converging pointwise on~$D$. In particular, for every $x\in D$ there exists $n(x)$ such that
\begin{align}\label{eq:aa1}
|F_n(x)-F_m(x)|\le \delta/3, \quad m,n\ge n(x).
\end{align}
For ${x_1}\in D$, \eqref{eq:equicontinuity} yields an open neighborhood $O_{x_1}$ with
\begin{align}\label{eq:aa2}
|F_n(x_1)-F_n(x_2)|\le \omega(x_1,d(x_1,x_2))+h_n\le \delta/6+\delta/6=\delta/3, \quad x_2\in O_{x_1},
\end{align}
for all $n\ge n_0$. Let $K\subset M$ be compact and $D'\subseteq D$ a finite set such that $\bigcup_{x'\in D'} O_{x'}$ covers $K$. Choose $n_{1}:= \max_{x'\in D'} n(x')\vee n_0$, then as any $x\in K$ is contained in an open neighborhood $O_{x'}$ of some $x'\in D'$, we obtain from \eqref{eq:aa1} and \eqref{eq:aa2} that
\begin{align*}
|F_n(x)-F_m(x)|\le |F_n(x)-F_n(x')|+|F_n(x')-F_m(x')|+|F_m(x')-F_m(x)|\le \delta,
\end{align*}
for all $x\in K$ and $m,n\ge n_{1}$. Thus $(F_n)$ has a limit~$F$, uniformly on compacts. Passing to the limit in~\eqref{eq:equicontinuity} shows that $F$ is $\omega$-continuous.
\end{proof}

Recall that $c:\X\times\Y\to \R_{+}$ is continuous. If $\Y_{\mathrm{cpt}}\subset \Y$ is compact and $\omega(x,r):=\sup_{y \in \Y_{\mathrm{cpt}}, d(x,x')\leq r} |c(x,y)-c(x',y)|$, then $\omega$ is a modulus of continuity in the above sense. %
That motivates the following estimates.

\begin{lemma}\label{lem:general}
Fix $\delta\in (0,1)$ and $\epsilon>0$. There exist compact sets $\X_{\mathrm{cpt}}\subseteq \X, \Y_{\mathrm{cpt}}\subseteq \Y$ and measurable sets $A_{\eps}\subseteq \X_{\mathrm{cpt}}$, $B_{\eps}\subseteq \Y_{\mathrm{cpt}}$ with $\mu(A_{\eps}),\nu(B_{\eps}) \ge 1-\delta$ such that 
\begin{align*}
\left|f_\epsilon(x_1)-f_\epsilon(x_2)\right|&\le \sup_{y\in \Y_{\mathrm{cpt}}} \left|c(x_1, y)-c(x_2,y)\right| -\epsilon\log(1-\delta) \quad\mbox{for} \quad x_1,x_2\in A_{\eps},\\[.5em]
\left|g_\epsilon(y_1)-g_\epsilon(y_2)\right|&\le \sup_{x\in A_{\eps}} \left|c(x, y_1)-c(x,y_2)\right| -\epsilon\log(1-\delta)\\
&\le \sup_{x\in \X_{\mathrm{cpt}}} \left|c(x, y_1)-c(x,y_2)\right| -\epsilon\log(1-\delta)  \quad\mbox{for} \quad y_1, y_2\in B_{\eps}.
\end{align*}
\end{lemma}

\begin{proof}
Fix $\kappa\in (0,\delta)$, to be determined later. 
Choose compacts $\X_{\mathrm{cpt}}$ and $\Y_{\mathrm{cpt}}$ with
$\mu(\X_{\mathrm{cpt}})\ge 1-\kappa^2/2$ and $\nu(\Y_{\mathrm{cpt}})\ge 1-\kappa^2/2$,
then $\pi_{\epsilon}\in \Pi(\mu,\nu)$ implies
\begin{align}\label{eq:tightness}
\pi_{\epsilon}(\X_{\mathrm{cpt}} \times \Y_{\mathrm{cpt}})\ge 1-\kappa^2.
\end{align}
Consider the set $$A_{\eps}=\left\{x\in \X_{\mathrm{cpt}}:\ \int_{\Y_{\mathrm{cpt}}} e^{\frac{f_\epsilon(x)+g_\epsilon(y)-c(x,y)}{\epsilon}}\,\nu(dy) \ge 1-\kappa\right\};$$
we claim that its complement $A_{\eps}^c=\X\setminus A_{\eps}$ satisfies
\begin{align}\label{eq:markov}
p_\epsilon:=\mu\left(A_{\eps}^c\right)\le \kappa.
\end{align}
Indeed, \eqref{eq:one} yields
\begin{align}\label{eq:exponential}
\int_{\Y_{\mathrm{cpt}}} e^{\frac{f_\epsilon(x)+g_\epsilon(y)-c(x,y)}{\epsilon}}\,\nu(dy)\le \int e^{\frac{f_\epsilon(x)+g_\epsilon(y)-c(x,y)}{\epsilon}}\,\nu(dy)=1
\end{align}
and thus
\begin{align*}
1-\kappa^2 &\stackrel{\eqref{eq:tightness}}{\le}\pi_{\epsilon}(\X_{\mathrm{cpt}} \times \Y_{\mathrm{cpt}})= \int_{\X_{\mathrm{cpt}}}\int_{\Y_{\mathrm{cpt}}} e^{\frac{f_\epsilon(x)+g_\epsilon(y)-c(x,y)}{\epsilon}} \,\nu(dy)\mu(dx)\\
&\le \int_{A_{\eps}^c}\int_{\Y_{\mathrm{cpt}}} e^{\frac{f_\epsilon(x)+g_\epsilon(y)-c(x,y)}{\epsilon}} \,\nu(dy)\mu(dx)\\
&\quad +\int_{A_{\eps}}\int_{\Y_{\mathrm{cpt}}} e^{\frac{f_\epsilon(x)+g_\epsilon(y)-c(x,y)}{\epsilon}} \,\nu(dy)\mu(dx)\\
&\stackrel{\eqref{eq:exponential}}{\le} (1-\kappa)p_\epsilon+(1-p_\epsilon)=1-p_\epsilon\kappa,
\end{align*}
which implies~\eqref{eq:markov}. Next, we observe from the definition of $A_{\eps}$ and \eqref{eq:exponential} that for $x\in A_{\eps}$,
\begin{align}\label{eq:ineq1}
\begin{split}
-\epsilon\left( \log \int_{\Y_{\mathrm{cpt}}} e^{\frac{g_\epsilon(y)-c(x,y)} {\epsilon}}\,\nu(dy)-\log(1-\kappa)\right)&\le f_\epsilon(x)\\
&\le 
-\epsilon\log \int_{\Y_{\mathrm{cpt}}} e^{\frac{g_\epsilon(y)-c(x,y)} {\epsilon}}\,\nu(dy).
\end{split}
\end{align}
Let $x_1, x_2\in A_{\eps}$ and assume without loss of generality that $f_\epsilon(x_1)\ge f_\epsilon(x_2)$. Then
\begin{align}\label{eq:longest}
\begin{split}
\left|f_\epsilon(x_1)-f_\epsilon(x_2)\right|&\stackrel{\eqref{eq:ineq1}}{\le} \epsilon \left(\log \int_{\Y_{\mathrm{cpt}}} e^{\frac{g_\epsilon(y)-c(x_2,y)}{\epsilon}}\, \nu(dy)-\log(1-\kappa)\right)\\
&\quad -\epsilon \log \int_{\Y_{\mathrm{cpt}}} e^{\frac{g_\epsilon(y)-c(x_1,y)}{\epsilon}}\, \nu(dy) \\
&=\epsilon \log \int_{\Y_{\mathrm{cpt}}} e^{\frac{c(x_1,y)-c(x_2,y)+g_\epsilon(y)-c(x_1,y)}{\epsilon}}\, \nu(dy) -\epsilon\log(1-\kappa)\\
&\qquad-\epsilon \log \int_{\Y_{\mathrm{cpt}}} e^{\frac{g_\epsilon(y)-c(x_1,y)}{\epsilon}}\, \nu(dy)\\
&\le \epsilon \log \left( e^{\frac{\sup_{y\in \Y_{\mathrm{cpt}}} \left|c(x_1, y)-c(x_2,y)\right| }{\epsilon}} \int_{\Y_{\mathrm{cpt}}} e^{\frac{g_\epsilon(y)-c(x_1,y)}{\epsilon}}\, \nu(dy) \right) \\
&\qquad-\epsilon\log(1-\kappa)-\epsilon \log \int_{\Y_{\mathrm{cpt}}} e^{\frac{g_\epsilon(y)-c(x_1,y)}{\epsilon}}\, \nu(dy)\\
&=\sup_{y\in \Y_{\mathrm{cpt}}} \left|c(x_1, y)-c(x_2,y)\right| -\epsilon\log(1-\kappa).
\end{split}
\end{align}
This concludes the proof of the first estimate in the lemma. 

Turning to the second, note that by \eqref{eq:tightness}, \eqref{eq:markov} and the definition of $A_{\eps}$,
\begin{align}\label{eq:nextround}
\begin{split}
\pi_{\epsilon}(A_{\eps} \times \Y_{\mathrm{cpt}} )
&\ge  \pi_{\epsilon}(\X_{\mathrm{cpt}} \times \Y_{\mathrm{cpt}} )-\pi_{\epsilon}( \X\setminus A_{\eps} \times \Y_{\mathrm{cpt}} )\\
&\ge  1-\kappa^2-\int_{A_{\eps}^c} \int_{\Y_{\text{cpt}}} e^{\frac{f_\epsilon(x)+g_\epsilon(y)-c(x,y)}{\epsilon}} \,\nu(dy)\,\mu(dx)\\
&\ge  1-\kappa^2-\kappa(1-\kappa)
=1-\kappa=1-\delta^2,
\end{split}
\end{align}
where we chose $\kappa:=\delta^2$ (ensuring  $\kappa \in (0,\delta)$, in particular). Define \begin{align*}
B_{\eps}=\left\{y\in \Y_{\mathrm{cpt}}:\ \int_{A_{\eps}} e^{\frac{f_\epsilon(x)+g_\epsilon(y)-c(x,y)}{\epsilon}}\,\mu(dx) \ge 1-\delta\right\}.
\end{align*}
Arguing as for~\eqref{eq:markov} and~\eqref{eq:ineq1}, now using \eqref{eq:nextround} instead of~\eqref{eq:tightness}, we see that $\nu(B_{\eps}^{c}) \le \delta$ and that for $y\in B_{\eps}$,
\begin{align*}%
\begin{split}
-\epsilon\left( \log \int_{A_{\eps}} e^{\frac{f_\epsilon(x)-c(x,y)} {\epsilon}}\,\mu(dx)-\log(1-\delta)\right)&\le g_\epsilon(y)\\
&\le 
-\epsilon\log \int_{A_{\eps}} e^{\frac{f_\epsilon(x)-c(x,y)} {\epsilon}}\,\mu(dx).
\end{split}
\end{align*}
We conclude the proof by arguing as in~\eqref{eq:longest} but with $f_\epsilon,\kappa$ replaced by $g_\epsilon,\delta$.
\end{proof}

The following extension lemma is a variation on Kirszbraun's theorem. Recall that a pseudometric $\tilde{d}$ is defined like a metric except that $\tilde{d}(x,y)=0$ need not imply $x=y$.

\begin{lemma}\label{lem:extension}
Let $(M,\tilde{d})$ be a pseudometric space and $A\subseteq M$.
Let $F:A\to \R$ satisfy 
\begin{align}\label{eq:continuity3}
|F(x_1)-F(x_2)|\le \tilde{d}(x_1,x_2)+\gamma ,\quad x_1,x_2\in A,
\end{align}
for some $\gamma>0$. Then the function $\tilde{F}:M\to \R$ defined by
\begin{align*}
\tilde{F}(x):=\inf_{x'\in A} \left[F(x')+\tilde{d}(x,x')+\gamma\1_{\{x'\neq x\}}\right], \quad x\in M
\end{align*}
satisfies $\tilde{F}=F$ on $A$ and
\begin{align}\label{eq:extension}
|\tilde{F}(x_1)-\tilde{F}(x_2)|\le \tilde{d}(x_1,x_2)+\gamma, \quad x_1,x_2\in M.
\end{align}
\end{lemma}

\begin{proof}
Fix $x\in  A$. For $x\neq x'\in A$ we have by \eqref{eq:continuity3} that
\begin{align*}
F(x')+\tilde{d}(x,x')+\gamma &\ge F(x)-\tilde{d}(x,x')-\gamma+\tilde{d}(x,x')+\gamma= F(x).
\end{align*}
It follows that $\tilde{F}(x)=F(x)$, showing the first claim. Fix $\kappa>0$ and let $x_1,x_2\in M$. By the definition of $\tilde{F}(x_2)$ there exists $x'\in A$ such that
$\tilde{F}(x_2)\ge F(x')+\tilde{d}(x_2,x')-\kappa$, and now the definition of $\tilde{F}(x_1)$ yields
\begin{align*}
\tilde{F}(x_1)-\tilde{F}(x_2)&\le  F(x')+\tilde{d}(x_1,x')+\gamma-F(x')-\tilde{d}(x_2,x')+\kappa\\
&=\tilde{d}(x_1,x')-\tilde{d}(x_2,x')+\gamma+\kappa
\le \tilde{d}(x_1,x_2)+\gamma+\kappa.
\end{align*}
As $\kappa>0$ was arbitrary, \eqref{eq:extension} follows.
\end{proof}

The next two lemmas show that limits of $f_{\eps},g_{\eps}$ must be Kantorovich potentials.

\begin{lemma}\label{lem:as}
Let $\eps_{n}\to0$ and suppose that the corresponding potentials $f_{\eps_{n}},g_{\eps_{n}}$ converge a.s. Then the limits $f:=\lim_{n} f_{\eps_{n}}$ and $g:=\lim_{n} g_{\eps_{n}}$ satisfy
\begin{align*}
f(x)+g(y) \le c(x,y) \qquad \mu\otimes\nu\text{-a.s.}
\end{align*}
\end{lemma}

\begin{proof}
Let $\delta>0$. Passing to a subsequence if necessary, we may assume that $\sum_{n=1}^\infty e^{-\delta/\epsilon_n}<\infty$.
Define
\begin{align*}
A_{\delta,n}=\left\{(x,y) \ : \ f_{\eps_{n}}(x)+g_{\eps_{n}}(y)-c(x,y) \ge \delta \right\},
\end{align*}
then
\begin{align*}
\begin{split}
1&\stackrel{\eqref{eq:one}}{=}\int e^{\frac{f_{\eps_{n}}(x)+g_{\eps_{n}}(y)-c(x,y)}{\epsilon_n}}\, \mu(dx)\nu(dy) \ge \int_{A_{\delta,n}} e^{\frac{f_{n}(x)+g_{n}(y)-c(x,y)}{\epsilon_n}}\, \mu(dx)\nu(dy) \\
&\ge e^{\delta/\epsilon_n} (\mu\otimes \nu)(A_{\delta,n})
\end{split}
\end{align*}
yields $(\mu\otimes\nu)(A_{\delta,n})\le e^{-\delta/\epsilon_n}$ and thus $\sum_{n} (\mu\otimes\nu)(A_{\delta,n})<\infty$.
The Borel--Cantelli lemma now shows that $(\mu\otimes\nu)(\limsup_{n} A_{\delta,n})=0$ and hence
\begin{align*}
f(x)+g(y) \le c(x,y)+\delta \qquad \mu\otimes\nu\text{-a.s.}
\end{align*}
As $\delta>0$ was arbitrary, the claim follows.
\end{proof}

\begin{lemma}\label{lem:lusin}
Let $c$ be upper semicontinuous and $f,g$ measurable functions with 
\begin{align*}
 f(x)+g(y) \le c(x,y)\quad \mu\otimes\nu\text{-a.s.}
\end{align*}
Then there exist versions $\tilde{f}=f$ $\mu$-a.s.\ and $\tilde{g}=g$ $\nu$-a.s.\ such that 
$$
  \tilde{f}(x)+\tilde{g}(y)\le c(x,y) \quad \text{for all}\quad(x,y)\in \X\times \Y.
$$
\end{lemma}

\begin{proof}
  Suppose first that $f,g$ are continuous. If $f(x)+g(y) > c(x,y)$ for some $(x,y)\in \X\times \Y$, the same inequality holds on a neighborhood $B_{r}(x)\times B_{r}(y)$, which then must be $\mu\otimes\nu$-null by the assumption. That is, $(x,y)\notin \spt\mu\times\spt\nu$. In conclusion, we can set $\tilde f=f$ on $\spt\mu$ and $\tilde f=-\infty$ outside $\spt\mu$, and similarly for~$\tilde g$.
  
  In general, Lusin's theorem yields an increasing sequence of closed sets $A_{n}\subset\X$ such that $f|_{A_{n}}$ is continuous  and $\mu(A_{n}^{c})\leq 1/n$. Let $\mu_{n}=\mu|_{A_{n}}$ and $A'_{n}=\spt \mu_{n}$. Defining analogously $B'_{n}\subset \Y$, the above argument shows that $f(x)+g(y) \leq c(x,y)$ on $A'_{n}\times B'_{n}$. The same inequality then holds on the product of $\cup_{n}A'_{n}$ and $\cup_{n}B'_{n}$, and these sets have full measure. It remains to set $\tilde f=f$ on $\cup_{n}A'_{n}$ and $\tilde f=-\infty$ on the complement, and similarly for~$\tilde g$.
\end{proof}

\section{Proof of the Main Result}\label{se:ProofOfMainRes}

We can now report the proof of Theorem~\ref{thm:1}. To simplify the notation, let us agree that an index~$n$ always refers to an object associated with $\eps=\eps_{n}$; for instance, $f_{n}=f_{\eps_{n}}$ and $g_{n}=g_{\eps_{n}}$. Moreover, subsequences are not relabeled. 

Steps 1--5 below establish the a.s.\ convergence of $f_n$ and $g_{n}$ along a subsequence. The final Step~6 shows that a.s.\ convergence also implies $L^{1}$-convergence, and that limits are Kantorovich potentials. %

\begin{proof}[Proof of Theorem \ref{thm:1}]
Let $\eps_{n}\to0$. In addition, we fix a strictly decreasing sequence $\delta_{k}\to 0$ with $\delta_{k}< 1/2$.\\[-.7em]

\textit{Step 1.}
For each $k, n\in \N$, Lemma \ref{lem:general} yields sets $$A_{n}(\delta_k)\subseteq \X_{\mathrm{cpt}}(\delta_k)\subseteq \X\quad \text{and}\quad B_{n}(\delta_k)\subseteq \Y_{\mathrm{cpt}}(\delta_k)\subseteq \Y$$ such that $$\mu(A_{n}(\delta_k))\ge 1-\delta_k\quad \text{and}\quad \nu (B_{n}(\delta_k))\ge 1-\delta_k$$%
as well as
\begin{align}\label{eq:continuity}
\begin{split}
\left|f_n(x_1)-f_n(x_2)\right|&\le \sup_{y\in \Y_{\mathrm{cpt}}(\delta_k)} \left|c(x_1, y)-c(x_2,y)\right| -\epsilon_n\log(1-\delta_k)\\
&\le \sup_{y\in \Y_{\mathrm{cpt}}(\delta_k)} \left|c(x_1, y)-c(x_2,y)\right| +\epsilon_n\log(2)
\end{split}
\end{align}
for $x_1,x_2\in A_{n}(\delta_k)$ and similarly 
\begin{align*}
\left|g_n(y_1)-g_n(y_2)\right|
&\le \sup_{x\in \X_{\mathrm{cpt}}(\delta_k)} \left|c(x, y_1)-c(x,y_2)\right| +\epsilon_n\log(2)
\end{align*}
for $y_1, y_2\in B_{n}(\delta_k)$. For each $n$, we can assume that the sequences  $(\X_{\mathrm{cpt}}(\delta_k))_{k}$ and $(\Y_{\mathrm{cpt}}(\delta_k))_{k}$ are increasing, and consequently also that $(A_{n}(\delta_k))_{k}$ and $(B_{n}(\delta_k))_{k}$ are increasing.\\[-.7em]

\textit{Step 2.} Define
\begin{align*}
\tilde{d}_k(x_1,x_2):= \sup_{y\in \Y_{\mathrm{cpt}}(\delta_k)} \left|c(x_1, y)-c(x_2,y)\right|.
\end{align*}
It is elementary to verify that $\tilde{d}_k$ is a pseudometric on~$\X$.
Using~\eqref{eq:continuity} and Lemma~\ref{lem:extension} with $\gamma=\epsilon_n\log(2)$, there exists an extension $f^k_n$ satisfying $f^k_n=f_n$ on $A_{n}(\delta_{k})$ and
\begin{align}\label{eq:extension1}
\left|f^k_n(x_1)-f^k_n(x_2)\right|\le \tilde{d}_k(x_1,x_2) +\epsilon_n\log(2), \quad x_1,x_2\in \X.
\end{align}
Similarly, there exists an extension $g^k_n$ for $g_{n}$ with an analogous property.\\[-.7em]

\textit{Step 3.} We now vary $n$, while still keeping $k$ fixed, and our aim is to construct a subsequential limit $f^k=\lim_{n\to \infty} f_n^k$ $\mu$-a.s. We first argue that $(f^k_n)_{n\in \N}$ is pointwise bounded from above. Indeed, after taking another subsequence if necessary, there exists $x_0\in \spt\mu$ such that $x_0\in A_{n}(\delta_k)$ for all $n$ and $f_n(x_0)\leq \int c(x_0,y)\,\nu(dy)<\infty$; cf.\ Lemma~\ref{lem:1}. Thus $f_n(x_0)^{+}$ is bounded uniformly in~$n$. On the other hand, 
$
\int f_n^+(x)\,\mu(dx)\le \int c(x,y)\,\nu(dy)\mu(dx)<\infty,
$
and as $\int f_n(x)\,\mu(dx)\ge 0$ by~\eqref{eq:normalization}, it follows that $\int f_n^-(x)\,\mu(dx)$ is bounded. In view of~\eqref{eq:continuity}, we obtain that $f_n(x_0)^{-}$ is bounded uniformly in~$n$. This shows that $f_n(x_0)$ is bounded, and then so is $f^{k}_n(x_0)$. By~\eqref{eq:extension1}, it follows that $f^{k}_n(x)$ is bounded uniformly in~$n$, for any $x\in\X$, as claimed. %

Define
$$
  \omega_{k}(x,r) = \sup_{d(x,x')\leq r}  \tilde{d}_{k}(x,x') \equiv  \sup_{y\in \Y_{\mathrm{cpt}}(\delta_k), d(x,x')\leq r} \left|c(x, y)-c(x',y)\right|.
$$
Clearly $\tilde{d}(x_1,x_2) \le \omega(x_{1},d(x_1,x_2))$, and $\omega_{k}$ is a modulus of continuity as noted above. In particular, the conditions of Lemma~\ref{lem:aa1} are satisfied for the sequence $(f^k_n)_{n\in\N}$ with $\omega:=\omega_{k}$ and $h_n:=\epsilon_n\log(2)$. After passing to a subsequence, we thus obtain an $\omega_{k}$-continuous function $f^k$ such that $f^k_n\to f^{k}$ uniformly on compact subsets. After passing to another subsequence, we similarly obtain a limit~$g^{k}$ for~$g^k_n$.

Recall that for fixed $n$, the sets $A_{n}(\delta_k)$ are increasing in~$k$, and $\cup_{k}A_{n}(\delta_k)$ has full $\mu$-measure. As a consequence, $f^k_n=f^{k'}_n=f_n$ on $A_{n}(\delta_{k})$ for all $k'\ge k$, and a diagonal argument yields a subsequence along which $\lim_{n\to \infty} f_n^k= f^k$ $\mu$-a.s.\ for all~$k$. Similarly for $g_n^k$, and we may assume in what follows that $\lim_{n\to \infty} f_n^k= f^k$ $\mu$-a.s.\ and $\lim_{n\to \infty} g_n^k= g^k$ $\nu$-a.s.\ for all $k$.\\[-.7em]

\textit{Step 4.} In this step we show that $(f^k)$ converges $\mu$-a.s., after passing to a subsequence. Fix $\gamma>0$ and choose $k_0$ such that $\delta_{k_0}\le \gamma$. For all $k,k'\ge k_0$ and all~$n$, we have
\begin{align}\label{eq:triangle}
|f^k(x)-f^{k'}(x)| &\le  |f^k(x)-f_n^k(x)|+|f_n^k(x)-f_n^{k'}(x)|+|f_n^{k'}(x)-f^{k'}(x)|.
\end{align}
Recalling also $f^k_n=f^{k'}_n=f_n$ on $A_{n}(\delta_{k_0})$ and
$\mu( (A_{n}(\delta_{k_0}))^c)\le \delta_{k_0}\le \gamma$,
we deduce
\begin{align*}
\int [ |f^k(x)&-f^{k'}(x)|\wedge 1]\,\mu(dx)\\
& \le \int_{A_{n}(\delta_{k_0})} \Big( [|f^k(x)-f_n^k(x)|\wedge 1] 
+ [|f_n^{k'}(x)-f^{k'}(x)|\wedge 1] \Big) \,\mu(dx)+\gamma.
\end{align*}
Sending $n\to \infty$ and using the result of Step~3, dominated convergence allows us to conclude that
$\int |f^k(x)-f^{k'}(x)|\wedge 1 \,\mu(dx)\le \gamma;$
that is, $(f^k)$ is Cauchy in $\mu$-probability. In particular, there exists a limit~$f$ in $\mu$-probability, and after taking a subsequence, the limit also holds $\mu$-a.s. Similarly, $\lim_{k}g^k =g$ $\nu$-a.s. \\[-.7em]

\textit{Step 5.} Next,  we show that the potentials $f_{n},g_{n}$ converge a.s.\ to the same limits $f,g$, after taking another subsequence. Given $\gamma>0$, Step~4 implies that for a.e.\ $x\in \X$ there exists $k_{0}(x)$ such that 
$|f^{k}(x)-f(x)|\le \gamma/3$ and $\delta_k\le \gamma$ 
for~$k\geq k_{0}(x)$.
As $\lim_{n} f_{n}^{k}=f^{k}$ $\mu$-a.s., it follows for $k\geq k_{0}(x)$ and for~$n$ sufficiently large that
\begin{align*}
|f_n(x)-f(x)|&\le |f_n(x)-f_n^k(x)|+|f_n^k(x)-f^k(x)|+|f^k(x)-f(x)|\\
&\le |f_n(x)-f_n^k(x)|+|f_n^k(x)-f^k(x)|+\gamma/3\\
&\le  |f_n(x)-f_n^k(x)|+\gamma/2.
\end{align*}
Recalling that $f_n(x)=f_n^k$ on $A_{n}(\delta_k)$, we conclude
$$
\lim_{n\to \infty} \mu \left(\left\{ x:\ |f_n(x)-f(x)|\ge \gamma\right\}\right)\le\limsup_{n\to \infty} \mu\left(\left(A_{n}(\delta_k)\right)^c\right)\le \delta_k\le \gamma;
$$
that is, $f_n \to f$ in $\mu$-probability. Taking another subsequence, we have $\lim_{n} f_n =f$ $\mu$-a.s. Similarly, we obtain $\lim_{n}g_n=g$.
Lemmas~\ref{lem:as} and~\ref{lem:lusin} 
show that after modifying $f,g$ on nullsets, we have 
\begin{align}\label{eq:inequality}
 f(x)+g(y)\le c(x,y), \quad (x,y)\in \X\times\Y.
\end{align}

\vspace{.3em}
\textit{Step 6.}
Let $C_{1}(x):=\int c(x,y)\,\nu(dy)$ and $C_{2}(y):=\int c(x,y)\,\mu(dx)$. In view of Lemma~\ref{lem:1} we have
\begin{equation}\label{eq:upperBoundInt}
  f_n,f \leq C_{1} \in L^{1}(\mu), \qquad g_n,g \leq C_{2} \in L^{1}(\nu).
\end{equation}
Using also $H(\pi|\mu\otimes\nu)\ge 0$, the duality $I_{\eps}=S_{\eps}$ from the Introduction, Fatou's lemma and \eqref{eq:inequality}, we obtain
\begin{align*}
\inf_{\pi\in \Pi(\mu,\nu)} \int c(x,y)\,\pi(dx,dy)&\le \lim_{n\to \infty}\left( \inf_{\pi\in \Pi(\mu,\nu)} \int c(x,y)\,\pi(dx,dy) +\epsilon_n H(\pi |  \mu\otimes\nu)\right)\\
&=\lim_{n\to \infty} \left(\int f_n(x)\, \mu(dx)+\int g_n(y) \,\nu(dy)\right)\\
&\le \int \limsup_{n\to \infty} f_n(x)\, \mu(dx)+\int \limsup_{n\to \infty} g_n(y) \,\nu(dy)\\
&=\int f(x)\, \mu(dx)+\int g(y) \,\nu(dy)\\
&=\inf_{\pi\in \Pi(\mu,\nu)} \int [f(x) + g(y)]\,\pi(dx,dy)\\
&\le \inf_{\pi\in \Pi(\mu,\nu)} \int c(x,y)\,\pi(dx,dy).
\end{align*}
In particular, $\lim_{\eps} S_{\eps}=\int f(x)\, \mu(dx)+\int g(y) \,\nu(dy) = S_{0}$. Using again~\eqref{eq:upperBoundInt}, Fatou's lemma then also shows that 
$$S_{0}/2 = \lim_{\eps\to0} S_{\eps}/2 = \lim_{n\to\infty} \int f_n(x)\, \mu(dx) \leq \int f(x)\, \mu(dx)$$
and similarly $S_{0}/2 \leq \int g(y)\, \nu(dy)$. We conclude that $$\int f(x)\, \mu(dx)=\int g(y)\, \nu(dy)=S_{0}/2$$
and hence the separate convergence
$$
  \lim_{n\to\infty} \int f_n(x)\, \mu(dx) = \int f(x)\, \mu(dx), \quad \lim_{n\to\infty} \int g_n(x)\, \mu(dx) = \int g(x)\, \mu(dx).
$$
In view of~\eqref{eq:upperBoundInt} and the a.s.\ convergence $f_{n}\to f$, applying Scheff\'e's lemma to the nonpositive sequence $f_{n}-C_{1}$ allows us to conclude that $f_{n}\to f$ in $L^{1}(\mu)$. Similarly, $g_{n}\to g$ in $L^{1}(\nu)$.
\end{proof}

The proof of Theorem~\ref{thm:1} simplifies substantially if $c$ is uniformly continuous (and in particular if~$\X$ and~$\Y$ are compact). Moreover, the conclusion is stronger  in this case: the almost-sure convergence of $f_{n}\to f$ and $g_{n}\to  g$  can be replaced by uniform convergence on compact subsets. For the remainder of this section, let $\omega: \R_{+}\to\R_{+}$ be a modulus of continuity as defined before Lemma~\ref{lem:aa1}.

\begin{lemma}\label{le:unifCont}
Suppose that $c$ is uniformly $\omega$-continuous in both variables.
Then the potentials $f_\epsilon,g_\epsilon$ %
are uniformly $\omega$-continuous, for any $\eps>0$.
\end{lemma}

\begin{proof}
Let $x_1, x_2\in \X$ satisfy $f_\epsilon(x_1)\ge f_\epsilon(x_2)$. Then
\begin{align*}
| & f_\epsilon(x_1)-f_\epsilon(x_2) |\\
&=\epsilon \log \int e^{\frac{g_\epsilon(y)-c(x_2,y)}{\epsilon}}\, \nu(dy)-\epsilon \log \int e^{\frac{g_\epsilon(y)-c(x_1,y)}{\epsilon}}\, \nu(dy) \\
&=\epsilon \log \int e^{\frac{c(x_1,y)-c(x_2,y)+g_\epsilon(y)-c(x_1,y)}{\epsilon}}\, \nu(dy)  -\epsilon \log \int e^{\frac{g_\epsilon(y)-c(x_1,y)}{\epsilon}}\, \nu(dy)\\
&\le \epsilon \log \left( e^{\frac{\sup_{y\in \Y} |c(x_1,y)-c(x_2,y)|}{\epsilon}} \int e^{\frac{g_\epsilon(y)-c(x_1,y)}{\epsilon}}\, \nu(dy) \right)
-\epsilon \log \int e^{\frac{g_\epsilon(y)-c(x_1,y)}{\epsilon}}\, \nu(dy)\\
&=\sup_{y\in \Y} |c(x_1,y)-c(x_2,y)| \leq \omega(d(x_{1},x_{2})).
\end{align*}
The case $f_\epsilon(x_1)\le f_\epsilon(x_2)$ follows by symmetry and the proof for $g_\epsilon$ is analogous.
\end{proof}

\begin{proposition} \label{pr:unifContCase}
Let $c$ be uniformly $\omega$-continuous in both variables and $\epsilon_n\to0$. After passing to a subsequence, $f_{\eps_{n}}\to f$ and $g_{\eps_{n}}\to g$ uniformly on compact subsets, for some uniformly $\omega$-continuous Kantorovich potentials $f$ and $g$.
\end{proposition}

\begin{proof}
The functions $(f_n)$ are $w$-equicontinuous by Lemma~\ref{le:unifCont}, hence $(f_n)$ is pointwise bounded as soon as it is bounded at one point $x\in\X$. By~Lemma \ref{lem:1}, $f_n(x)\le \int c(x,y)\,\nu(dy)<\infty$ for $\mu$-a.e.\ $x$, so that $(f_n^{+})$ is pointwise bounded. On the other hand, 
$
\int f_n^+(x)\,\mu(dx)\le \int c(x,y)\,\nu(dy)\mu(dx)<\infty,
$
and as $\int f_n(x)\,\mu(dx)\ge 0$ by~\eqref{eq:normalization}, it follows that $\int f_n^-(x)\,\mu(dx)$ is bounded. By equicontinuity, it follows that $(f_n^-)$ must be bounded at any point $x\in \spt\mu$, and then at all points. Similarly for $(g_{n})$, and now the claimed convergence to some uniformly $\omega$-continuous functions $f,g$ follows from the Arzel\`a--Ascoli theorem. To see that $f,g$ are Kantorovich potentials, we argue as in Step~6 of the proof of Theorem~\ref{thm:1}.
\end{proof}

\section{Relation to a Large Deviations Principle}\label{se:LDP}

In this section we discuss the connection between convergence of potentials (Theorem~\ref{thm:1}) and a large deviations principle (LDP) along the lines proposed in \cite[Theorem~1.1]{BerntonGhosalNutz.21}.  
Roughly speaking, the LDP describes the exponential rate of decay of $\pi_{\eps}(E)$ for a set $E$ outside the support of $\pi_{0}:=\lim_{\eps\to0} \pi_{\eps}$, whereas the convergence %
of potentials yields the exponential rate of decay of the density $\pi_{\eps}/d(\mu\otimes\nu)$ at points outside of $\{f+g =c\}$. Clearly, these statements are closely related, and as seen below, they are equivalent if $\X,\Y$ are compact. In the general case, however, neither result implies the other in an obvious way.

Throughout this section, we fix a sequence $\eps_{n}\to0$ and set $(f_{n},g_{n}):=(f_{\eps_{n}},g_{\eps_{n}})$, as in Section \ref{se:ProofOfMainRes}. Given a measurable function $I$ on $\X\times\Y$, we denote by $\essinf I$ the essential infimum wrt.\ $\mu\otimes\nu$, defined as 
$
  \essinf I = \inf \{\alpha\in\R:\, (\mu\otimes\nu)\{I < \alpha\}>0\}.
$

\begin{proposition}\label{pr:LDPfromConv}
  Suppose that $f:=\lim_{n}f_{n}$ exists in $\mu$-probability and $g:=\lim_{n}g_{n}$ exists in $\nu$-probability. Define $I(x,y):= c(x,y)-f(x)-g(y)$, then for any measurable set $E\subset \X\times\Y$,
  \begin{align}\label{eq:LDPfromConv}
    \liminf_{n\to \infty} \epsilon_n \log \pi_{\epsilon_n}(E) &\ge - \essinf_{(x,y)\in E}   I(x,y).
  \end{align}
  If the convergence of $(f_{n},g_{n})$ is a.s.\ uniform on~$E$; i.e.,
  $$
   \|(f_{n},g_{n})\1_{E}-(f,g)\1_{E}\|_{L^\infty(\mu\otimes\nu)} \to 0,
  $$
  then~$E$ also satisfies the matching bound
  \begin{align}\label{eq:LDPfromConv2}
    \limsup_{n\to \infty} \epsilon_n \log \pi_{\epsilon_n}(E) &\le - \essinf_{(x,y)\in E}   I(x,y).
  \end{align}
  If $\X,\Y$ are compact, that is the case for all sets~$E$.
\end{proposition} 

\begin{proof}
Let $E\subset \X\times\Y$ be measurable,
$
\alpha:= - \essinf_{(x,y)\in E}   I(x,y) 
$
and $\gamma>0$. By the definition of $\alpha$, the set
\begin{align*}
E^{\gamma}:= \{ (x,y)\in E:\ f(x)+g(y)-c(x,y) \ge \alpha-2\gamma\}
\end{align*}
satisfies $\beta:=(\mu \otimes \nu)(E^\gamma)/2>0$. In view of the assumed convergence of $(f_{n},g_{n})$,  there exists $n_{0}$ such that
$(\mu\otimes\nu)\{|(f_{n},g_{n})-(f,g)|>\gamma\} \leq   \beta$ for $n\geq n_{0}$, so that 
$$
  E^{\gamma}_{n}:= \{(x,y)\in E^{\gamma}:\, f_{n}(x)+g_{n}(y)-c(x,y) \ge \alpha-\gamma\}  
$$
satisfies $(\mu \otimes \nu)(E^{\gamma}_{n})\ge \beta$ for all $n\geq n_{0}$. Thus 
\begin{align*}
\pi_{\epsilon_n} (E)&\ge \pi_{\epsilon_n} (E^{\gamma}_{n})
= \int_{E^{\gamma}_{n}} e^{\frac{f_n(x)+g_n(y)-c(x,y)}{\epsilon_n}} \,\mu(dx)\,\nu(dy)
\ge \beta e^{\frac{\alpha-\gamma}{\epsilon_n}}
\end{align*}
for $n\geq n_{0}$ and then $\liminf_{n\to \infty} \epsilon_n \log \pi_{\epsilon_n} (E) \ge\alpha-\gamma$. As $\gamma>0$ was arbitrary, the lower bound~\eqref{eq:LDPfromConv} follows.
Turning to the second claim, note that
\begin{align*}
\epsilon_n \log \pi_{\epsilon_n}(E) 
&=\epsilon_n \log\left( \int_E e^{\frac{f_n(x)+g_n(y)-c(x,y)}{\epsilon_n}} \,\mu(dx)\,\nu(dy)\right) \\
&\le \esssup_{(x,y)\in E} \left( f_n(x)+g_n(y)-c(x,y)\right).
\end{align*}
If $\|(f_{n},g_{n})\1_{E}-(f,g)\1_{E}\|_{\infty}\to0$, it readily follows that
\begin{align*}
  \lim_{n\to \infty} \epsilon_n \log \pi_{\epsilon_n}(E) &\le - \essinf_{(x,y)\in E}   I(x,y),
\end{align*}
as desired. If $\X,\Y$  are compact, Proposition~\ref{pr:unifContCase} and the  assumed convergence of the potentials in probability imply that $\|(f_{n},g_{n})-(f,g)\|_{\infty}\to0$ (without taking a subsequence), so that the above applies to any measurable set~$E$. 
\end{proof} 

\begin{remark} 
  (a) In Proposition~\ref{pr:LDPfromConv} the rate is stated through an essential infimum, consistent with the fact that~$E$ can be irregular and $f,g$ are considered as determined only up to nullsets. In many situations it is known that Kantorovich potentials admit a continuous version, for instance by $c$-concavity. If moreover~$E$ is suitably regular (e.g., open and contained in $\spt \mu \times\spt \nu$), the essential infimum can be written as an infimum.
  
  (b) In the case of compactly supported marginals, an alternative proof of Proposition~\ref{pr:LDPfromConv} can be given using Bryc's inverse to Varadhan's Integral Lemma; cf.\ \cite[Theorem~4.4.2]{DemboZeitouni.10}. That proof, however, is longer than the direct argument given above. In connection with classical large deviations theory, we note that the sequence $(\pi_{\eps_{n}})$ fails to be exponentially tight whenever the marginals are not compactly supported: exponential tightness implies, in particular, that any limit $\pi_{0}$ is compactly supported, but as $\pi_{0}\in\Pi(\mu,\nu)$, the same then follows for $\mu,\nu$.

\end{remark}

If the Kantorovich potentials $(f,g)$ are unique, Theorem~\ref{thm:1} implies that the first condition in Proposition~\ref{pr:LDPfromConv} is satisfied.

Bounds similar to~\eqref{eq:LDPfromConv} and~\eqref{eq:LDPfromConv2} are stated in \cite[Theorem~1.1]{BerntonGhosalNutz.21} for open and compact sets, respectively. While weak convergence $\pi_{\eps_{n}}\to\pi_{0}$ of the couplings is assumed, it avoids any conditions on $\X,\Y$, the integrability of~$c$, or even the finiteness of the value $I_{\eps}$ in~\eqref{eq:EOT}. Such a setting does seem outside the scope of the methods used here.

In general, if convergence of potentials is not known a priori, Proposition~\ref{pr:LDPfromConv} implies non-matching bounds by maximizing or minimizing over all potentials as follows. Given a family $(I_{\lambda})$ of measurable functions, $I^{*}:=\esssup_{\lambda} I_{\lambda}$ denotes the essential supremum wrt.\ $\mu\otimes\nu$ in the sense of probability theory.\footnote{I.e., $I^{*}$ is the (a.s.\ unique) measurable function satisfying $I^{*}\geq I_{\lambda}$ a.s.\ for all~$\lambda$ and $I^{*}\leq J$ a.s.\ for any~$J$ satisfying $J\geq I_{\lambda}$ a.s.\ for all~$\lambda$. In other words, $I^{*}$ is the supremum in the lattice of measurable functions equipped with the a.s.\ order.} Similarly, $\essinf_{\lambda} I_{\lambda}$ is the essential infimum.

\begin{corollary}\label{co:LDPlower}
  Define $I^{*}:= \esssup_{f,g} I_{f,g}$, where $I_{f,g}(x,y):= c(x,y)-f(x)-g(y)$ and the supremum is taken over all Kantorovich potentials $(f,g)$. Similarly, define $I_{*}:= \essinf_{f,g} I_{f,g}$. Then
  \begin{align}\label{eq:LDPfromConvCor}
    \liminf_{n\to \infty} \epsilon_n \log \pi_{\epsilon_n}(E) &\ge - \essinf_{(x,y)\in E}   I^{*}(x,y)
  \end{align}  
  for any measurable set $E\subset \X\times\Y$. If $\X,\Y$ are compact, then also
  \begin{align}\label{eq:LDPfromConv2Cor}
    \limsup_{n\to \infty} \epsilon_n \log \pi_{\epsilon_n}(E) &\le - \essinf_{(x,y)\in E}   I_{*}(x,y).
  \end{align}
\end{corollary} 

\begin{proof}
  Passing to a subsequence, we may assume that the $\liminf$ on the left-hand side is a limit. After passing to another subsequence, Theorem~\ref{thm:1} yields that the Schr\"odinger potentials $(f_{n},g_{n})$ converge in~$L^{1}$ to some Kantorovich potentials~$(f,g)$, and then Proposition~\ref{pr:LDPfromConv} applies to $(f,g)$. As $\essinf_{(x,y)\in E} I_{f,g}(x,y) \leq \essinf_{(x,y)\in E}   I^{*}(x,y)$, the lower bound~\eqref{eq:LDPfromConvCor} follows. The proof of~\eqref{eq:LDPfromConv2Cor} is analogous.
\end{proof} 

\begin{remark} 
  The lower bound~\eqref{eq:LDPfromConvCor} is quite general, and seems to be novel. Except in the case of uniqueness for the Kantorovich potentials, no analogue is stated in~\cite{BerntonGhosalNutz.21}. On the other hand, the upper bound~\eqref{eq:LDPfromConv2Cor} is similar to the bound in~\cite[Theorem~1.1\,(a)]{BerntonGhosalNutz.21}. The latter is stated under the condition that $\pi_{\eps_{n}}$ converges but without any conditions on $\X,\Y$.
\end{remark}

The next result is a partial converse to Proposition~\ref{pr:LDPfromConv}. It suggests that if an LDP holds, then the Schr\"odinger potentials must converge (without passing to a subsequence) and the rate function must be determined by the limiting potentials. We prove this in the compact case via Varadhan's Integral Lemma, but we conjecture that the assertions remains valid in some generality.

\begin{proposition}\label{pr:LDPtoPotentialConv}
Let $\X,\Y$ be compact and suppose the assertion of the LDP \cite[Theorem~1.1]{BerntonGhosalNutz.21} holds for some function~$I:\X\times\Y\to\R_{+}$; that is, 
\begin{align}\label{eq:ldp1}
\limsup_{n\to \infty} \epsilon_n \log \pi_{\epsilon_n}(C) &\le  -\inf_{(x,y)\in C}  I(x,y) \quad\mbox{for $C\subset\X\times\Y$ compact},\\
 \liminf_{n\to \infty} \epsilon_n \log \pi_{\epsilon_n}(U) &\ge -\inf_{(x,y)\in U} I(x,y) \quad\mbox{for $U\subset \spt \mu \times\spt \nu$ open} \label{eq:ldp2}.
\end{align}
Then 
$$
 I(x,y) = c(x,y) - f(x)-g(y), \quad (x,y)\in  \spt \mu \times\spt \nu
$$
for some Kantorovich potentials $(f,g)$, and 
$$
  f=\lim_{n\to \infty} f_{n} \quad \mbox{uniformly on }\spt \mu, \qquad g=\lim_{n\to \infty} g_{n} \quad \mbox{uniformly on }\spt \nu.
$$
\end{proposition}

\begin{proof}
As $\X\times\Y$ is compact, $c$ is uniformly continuous and then $f_{n},g_{n}$ are uniformly equicontinuous; cf.\ Lemma~\ref{le:unifCont}. 
Fix $(x_0,y_0)\in \spt \mu \times\spt \nu$. Equicontinuity implies that given $\gamma>0$ there exists $r,n_{0}>0$ such that for all $n\geq n_{0}$,
$$
 | I(x_0,y_0)+f_n(x_0)+g_n(y_0)-c(x_0,y_0) - J_{n}(r) |\le \gamma \quad\mbox{for}
$$
$$
  J_{n}(r) := \epsilon_n \log \bigg( \int_{B_{r}(x_0,y_0)} e^{\frac{I(x,y)+f_n(x)+g_n(y)-c(x,y)}{\epsilon_n}} \,\mu(dx)\,\nu(dy)\bigg).
$$
 To show $\lim_{n\to \infty} [f_n(x_{0})+g_n(y_{0})]=c(x_{0},y_{0})-I(x_{0},y_{0})$, it therefore suffices to prove for all $r>0$ that
\begin{align}\label{eq:proofLDPtoPotentialConv}
\lim_{n\to \infty} J_{n}(r)=0.
\end{align} 

Next, we argue that~$I$ must be continuous. Indeed, after passing to a subsequence, Proposition~\ref{pr:LDPfromConv} shows that $I$ must be of the form $I=c-\tilde{f}-\tilde{g}$ on $\spt \mu \times\spt \nu$, for some (necessarily uniformly continuous) Kantorovich potentials $(\tilde{f},\tilde{g})$.
Moreover, we may assume that $\X=\spt\mu$ and $\Y=\spt\nu$, by shrinking the marginal spaces if necessary.
In brief, the LDP  \eqref{eq:ldp1}, \eqref{eq:ldp2} then holds for all closed sets~$C$ and open sets~$U$ in $\X\times\Y$ with the ``good'' rate function $I$. In this context,  Varadhan's Integral Lemma \cite[Theorem~4.3.1]{DemboZeitouni.10} states that
\begin{align}\label{eq:VaradhanLemma}
\lim_{n\to \infty}  \epsilon_n \log \left( \int e^{\frac{\phi(x,y)}{\epsilon_n}} \pi_{\epsilon_n}(dx,dy) \right)= \sup_{(x,y)\in \X\times \Y} (\phi(x,y)-I(x,y))
\end{align}
for any continuous function $\phi:\X\times\Y\to\R$ that satisfies the moment condition
\begin{align*}
\limsup_{n\to \infty}  \epsilon_n \log \left( \int e^{\frac{\gamma \phi(x,y)}{\epsilon_n}} \pi_{\epsilon_n}(dx,dy) \right)<\infty
\end{align*}
for some $\gamma>1$.
As the continuous function~$I$ is bounded on the  compact space $\X\times\Y$, this holds in particular for $\phi:=I$, for any~$\gamma>1$. Let $(x_0,y_0)\in \X\times\Y$ and $r>0$. Using ~\eqref{eq:VaradhanLemma} for $\phi=I$,
\begin{align*}
\limsup_{n\to \infty} J_{n}(r)
&\le \limsup_{n\to \infty} \epsilon_n \log \left( \int e^{\frac{I(x,y)}{\epsilon_n}} \pi_{\epsilon_n}(dx,y)\right)\\
&=\sup_{(x,y)\in \X\times \Y} (I(x,y)-I(x,y))=0.
\end{align*}
To show the converse inequality, consider a bounded continuous function $\phi$ with
$$
  \phi(x_0,y_0) = I(x_0,y_0), \qquad \phi \leq I \mbox{ on } B_{r}(x_0,y_0),\qquad \phi = -1 \mbox{ on } B^{c}_{r}(x_0,y_0).
$$
Then
\begin{align*}
\int_{B_{r}(x_0,y_0)} e^{\frac{I(x,y)+f_n(x)+g_n(y)-c(x,y)}{\epsilon_n}} \,\mu(dx)\,\nu(dy)
&\geq \int_{B_{r}(x_0,y_0)} e^{\frac{\phi(x,y)}{\epsilon_n}} \,\pi_{\eps_n}(dx,dy) \\ 
&\geq \int e^{\frac{\phi(x,y)}{\epsilon_n}} \,\pi_{\eps_n}(dx,dy) - e^{\frac{-1}{\epsilon_n}}
\end{align*}
and thus
\begin{align*}
\liminf_{n\to \infty} J_{n}(r)
&\ge \liminf_{n\to \infty} \epsilon_n \log \left( \int e^{\frac{\phi(x,y)}{\epsilon_n}} \,\pi_{\eps_n}(dx,dy) \right) \\
&=\sup_{(x,y)\in \X\times \Y} (\phi(x,y)-I(x,y))=0,
\end{align*}
where we  have used~\eqref{eq:VaradhanLemma}. This completes the proof of~\eqref{eq:proofLDPtoPotentialConv} and thus shows that $\lim_{n\to \infty} [f_n(x_{0})+g_n(y_{0})]=c(x_{0},y_{0})-I(x_{0},y_{0})$ for $(x_0,y_0)\in \spt \mu \times\spt \nu$. In view of the uniform equicontinuity, the convergence is even uniform on that set.

On the other hand, we have already shown in Proposition~\ref{pr:unifContCase} that $f_{n}\to f$ and $g_{n}\to g$ uniformly, after passing to a subsequence, for some Kantorovich potentials $f,g$. Thus $c-I=f+g$ on $\spt \mu \times\spt \nu$. It remains to argue that the original sequences $f_{n},g_{n}$ converge to $f,g$. Indeed, the rectangular form of $S:=\spt \mu \times\spt \nu$ implies that if  $f(x)+g(y)=\tilde f(x)+\tilde g(y)$ on $S$, then $\tilde f=f+a$ and $\tilde g=g-a$ for some $a\in\R$. Recalling our symmetric normalization for potentials, the claim follows.
\end{proof}

\section{Varying Costs}\label{se:varyingCost}

In this section we extend Theorem~\ref{thm:1} to cost functions that vary with $\eps$. The continuous cost $c$ will be used for the limiting Monge--Kantorovich transport problem, as before. In addition, we introduce a family of  cost functions $c_{\eps}:\X\times\Y\to\R_{+}$ for the regularized problems with~$\eps>0$. These functions are merely required to be measurable.

On the one hand, we are interested in the stability of Theorem~\ref{thm:1} with respect to the cost function. On the other hand, this section is motivated by the large deviations perspective on Schr\"odinger bridges; cf.\ \cite{Leonard.14}. Recall that 
\begin{equation}\label{eq:bridge}
   \pi_{\eps}=\argmin_{\pi\in \Pi(\mu,\nu)} H(\pi|R_{\eps}) \quad \mbox{for} \quad \frac{dR_{\eps}}{d(\mu\otimes\nu)} = \alpha_{\eps} e^{-c/\eps}
\end{equation}
where $\alpha_{\eps}$ is the normalizing constant. Theorem~\ref{thm:1} and its counterparts in Section~\ref{se:LDP} can be interpreted as consequences of the large deviations of $(R_{\eps})$ as~$\eps\to0$, whose rate is the function~$c$. More generally, this  rate function is shared by arbitrary measures $(R'_{\eps})$ with $-\eps \log \frac{dR'_{\eps}}{d(\mu\otimes\nu)} \to c$, and one may wonder if they give rise to a similar result.
This convergence is equivalent to setting $\frac{dR'_{\eps}}{d(\mu\otimes\nu)}=\alpha'_{\eps}e^{-c_{\eps}/\eps}$ for some function $c_{\eps}$ with~$c_{\eps}\to c$, and returning to the language of entropic optimal transport, it corresponds to the cost~$c_{\eps}$ under consideration. 

In what follows, we assume a common bound
\begin{equation}\label{eq:majorant}
  c_{\eps} \leq \bar c \quad\mbox{for all}\quad \eps>0
\end{equation}
for some function $\bar{c}(x,y) = \bar{c}_{1}(x)+ \bar{c}_{2}(y)$ with $\bar{c}_{1}\in L^{1}(\mu)$ and $\bar{c}_{2}\in L^{1}(\nu)$, and that
\begin{equation}\label{eq:cConv}
  c_{\eps} \to c \quad \mbox{uniformly on compact subsets as $\eps\to0$.}
\end{equation}
The modified entropic optimal transport problem then reads
\begin{equation}\label{eq:EOTvar}
  I_{\eps} :=\inf_{\pi\in\Pi(\mu,\nu)} \int_{\X\times\Y} c_{\eps}(x,y) \, \pi(dx,dy) + \eps H(\pi|\mu\otimes\nu).
\end{equation}
As before, it has a unique solution $\pi_{\eps}$, and we introduce the Schr\"odinger potentials through the formula
\begin{equation}\label{eq:densityFormVar}
\frac{d\pi_{\eps}}{d(\mu\otimes\nu)}(x,y) = \exp \left(\frac{f_{\eps}(x) +g_{\eps}(y) - c_{\eps}(x,y)}{\eps}\right)
\end{equation}
and the symmetric normalization~\eqref{eq:normalizationIntro}. The Monge--Kantorovich problem and its potentials are still based on the continuous cost $c$. While not required for the regularized problem with $\eps>0$, continuity of costs is important for $\eps=0$, including for the validity of~Theorem~\ref{thm:1} (see Example~\ref{ex:noConv}).

\begin{proposition}\label{pr:varyingCost}
  Let~\eqref{eq:majorant}, \eqref{eq:cConv} hold. Then the assertion of Theorem~\ref{thm:1} extends to the setting~\eqref{eq:EOTvar}, \eqref{eq:densityFormVar} of variable costs $(c_{\eps})$.
\end{proposition} 

\begin{proof}
  We only indicate the necessary changes to the proof of Theorem~\ref{thm:1}. First of all, we recall that the auxiliary results in Section~\ref{sec:auxiliaryResults} did not require continuity. Next, we go through the steps in Section~\ref{se:ProofOfMainRes}.

\emph{Step~1.} We change~\eqref{eq:continuity} to
\begin{align*}%
\begin{split}
\left|f_n(x_1)-f_n(x_2)\right|&\le \sup_{y\in \Y_{\mathrm{cpt}}(\delta_k)} \left|c_{n}(x_1, y)-c_{n}(x_2,y)\right| -\epsilon_n\log(1-\delta_k)\\
&\le \sup_{y\in \Y_{\mathrm{cpt}}(\delta_k)} \left|c(x_1, y)-c(x_2,y)\right| +\epsilon_n\log(2) + \eta_{n,k}
\end{split}
\end{align*}
where, due to the uniform convergence of $c_{\eps}$ on the compact set $\X_{\mathrm{cpt}}(\delta_k)\times\Y_{\mathrm{cpt}}(\delta_k)$, the constant $\eta_{n,k}$ satisfies $\lim_{n}\eta_{n,k}=0$ (for fixed $k$). The subsequent display for $g_{n}$ is changed analogously.

\emph{Step~2.}  Instead of~\eqref{eq:extension1} we now have
\begin{align*}%
\left|f^k_n(x_1)-f^k_n(x_2)\right|\le \tilde{d}_k(x_1,x_2) +\epsilon_n\log(2) + \eta_{n,k}, \quad x_1,x_2\in \X.
\end{align*}

\emph{Step~3.} In the arguments for the pointwise boundedness, simply replace~$c$ by~$\bar c$. In the application of Lemma~\ref{lem:aa1}, replace $h_n:=\epsilon_n\log(2)$ by $h_{n,k}:=\epsilon_n\log(2)+ \eta_{n,k}$. Note that the dependence on~$k$ does not cause any difficulty, as $k$ is fixed and $\lim_{n}h_{n,k}=0$ holds for each~$k$.

\emph{Steps~4,5.} No changes are necessary in these steps; note that~\eqref{eq:inequality} is based solely on the limiting cost function $c$ which is still assumed to be continuous.

\emph{Step 6.}
Define $C_{1}(x):=\int \bar{c}(x,y)\,\nu(dy)=\bar{c}_{1}(x) + \|c_{2}\|_{L^{1}(\nu)}$ and similarly $C_{2}(y):=\bar{c}_{2}(y) + \|c_{1}\|_{L^{1}(\mu)}$. Then we again have~\eqref{eq:upperBoundInt}. For the subsequent display, we now need to argue that
\begin{align}\label{eq:MKlimit}
\inf_{\pi\in \Pi(\mu,\nu)} \int c(x,y)\,\pi(dx,dy)&\le \lim_{n\to \infty} \inf_{\pi\in \Pi(\mu,\nu)} \int c_{n}(x,y)\,\pi(dx,dy).
\end{align}
Indeed, given $\gamma>0$, we can find a compact set $K=K_{1}\times K_{2}\subset \X\times\Y$ with
$$
\int_{K^{c}} \bar{c}(x,y)\,\pi(dx,dy) \leq \int_{K_{1}^{c}} \bar{c}_{1}(x)\,\mu(dx) + \int_{K_{2}^{c}} \bar{c}_{2}(y)\,\nu(dy) <\gamma.
$$
As $c_{n}\to c$ uniformly on~$K$, we also have $|c-c_{n}|\leq \gamma$ on~$K$ for $n\geq n_{0}$. Thus
\begin{align*}
  \left|\inf_{\pi\in \Pi(\mu,\nu)} \int c \,d\pi  - \inf_{\pi\in \Pi(\mu,\nu)} \int c_{n} \,d\pi \right| 
  &\leq \sup_{\pi\in \Pi(\mu,\nu)} \int |c-c_{n}| \,d\pi \\
  &\leq \sup_{\pi\in \Pi(\mu,\nu)} \int_{K }|c-c_{n}| \,d\pi + \int_{K^{c}} \bar{c} \,d\pi
  \leq 2\gamma
\end{align*} 
for $n\geq n_{0}$. This implies~\eqref{eq:MKlimit}, even with equality, and the remainder of the proof holds as stated without further changes.
\end{proof}

The following simple example shows that continuity of~$c$ is important for  the validity of Theorem~\ref{thm:1}.

\begin{example}\label{ex:noConv}
  Let $\mu=\nu$ be uniform on $\X=\Y=[0,1]$ and $c(x,y)=\1_{x\neq y}$. Then the Schr\"odinger potentials are $f_{\eps}=g_{\eps}\equiv 1/2$ for all~$\eps>0$ but the (unique) Kantorovich potentials are $f_{0}=g_{0}\equiv0$.
\end{example}

To put the example in a broader context, note that the entropic optimal transport problem~\eqref{eq:EOT} with $\eps>0$ remains unchanged if the cost function is altered on a $\mu\otimes\nu$-nullset, whereas the Monge--Kantorovich problem may very well change. If~$c$ is measurable and $\hat{c}$ is a continuous function with $\hat c=c$, Theorem~\ref{thm:1} thus implies that the entropic problem~\eqref{eq:EOT} with cost~$c$ converges to the Monge--Kantorovich problem with cost~$\hat c$ for~$\eps\to0$. Example~\ref{ex:noConv} is a particular case with $c(x,y)=\1_{x\neq y}$ and $\hat c\equiv 1$. For more general cost functions, one may conjecture that~\eqref{eq:EOT} converges to some form of upper envelope of the Monge--Kantorovich problem; we leave this question for future research.

\section{Multimarginal Optimal Transport}\label{sec:multi}

Instead of two marginals $\mu$ and $\nu$, we can generalize to an arbitrary number $N\in \N$ of marginals. Consider Polish probability spaces $(\X_{i},\mu_{i})$ for $i=1,\dots,N$ and let
$$
  \boldsymbol\mu(dx_{1},\dots,dx_{N}):= \mu_1(dx_1) \otimes \cdots \otimes\mu_N(dx_N) %
$$
denote the product measure. Moreover, let $c:\X_1\times \cdots\times \X_N \to \R_+$ be continuous with
$
  \int c \,d\boldsymbol\mu <\infty.
$
The entropic optimal transport problem generalizes directly to the set $\pi\in\Pi(\mu_{1},\dots,\mu_{N})$ of couplings,
\begin{equation}\label{eq:EOTmulti}
  I_{\eps} :=\inf_{\pi\in\Pi(\mu_{1},\dots,\mu_{N})} \int c \, \pi + \eps H(\pi|\boldsymbol\mu),
\end{equation}
and has a unique solution $\pi_{\eps}$ given by
\begin{equation}\label{eq:densityFormMulti}
\frac{d\pi_{\eps}}{d\boldsymbol\mu}(x_{1},\dots,x_{N}) = \exp \left(\frac{f^{1}_{\eps}(x_{1}) + \cdots + f^{N}_{\eps}(x_{N}) - c(x_{1},\dots,x_{N})}{\eps}\right)
\end{equation}
with $f^{i}_{\eps}\in L^{1}(\mu_{i})$. For $\eps=0$, we again recover the multimarginal optimal transport problem, whose dual now reads
\begin{align}\label{eq:dualOTmulti}
S_0 :=\sup_{f^{i}\in L^1(\mu_{i}),\, \sum_{i}f^{i}(x_{i}) \leq c(x_{1},\dots,x_{N})} \, \sum_{i=1}^{N} \int f^{i}(x_{i})\, \mu_{i}(dx_{i}).
\end{align}
We again normalize all potentials symmetrically.
Extending Theorem~\ref{thm:1}, we have the following result.

\begin{theorem}\label{thm:multi}
Let $(f^{1}_{\eps},\dots,f^{N}_{\eps})$ be the unique Schr\"odinger potentials for $\eps>0$.
\begin{itemize}
\item[(a)] Given $\eps_{n}\to0$, there is a subsequence $(\eps_{k})$ such that $f^{i}_{\eps_{k}}$ converges in $L^{1}(\mu_{i})$, for all $i=1,\dots,N$.
\item[(b)] If $\lim_{n} f^{i}_{\eps_{n}} = f^{i}$ $\mu_{i}$-a.s.\ for all $i=1,\dots,N$, then $(f^{1},\dots,f^{N})$ are Kantorovich potentials and the convergence also holds in $L^{1}(\mu_{i})$.
\end{itemize}
If the Kantorovich potentials $(f^{1}_{0},\dots,f^{N}_{0})$ for~\eqref{eq:dualOTmulti} are unique, then it follows that $\lim_{\eps\to0} f^{i}_{\eps} = f^{i}_{0}$ in $L^{1}(\mu_{i})$ for  $i=1,\dots,N$.
\end{theorem}

\begin{proof}
The arguments are exactly the same as in the proof of Theorem~\ref{thm:1}, and therefore omitted.
\end{proof}

\newcommand{\dummy}[1]{}


\begin{thebibliography}{10}

\bibitem{AguehCarlier.11}
M.~Agueh and G.~Carlier.
\newblock Barycenters in the {W}asserstein space.
\newblock {\em SIAM J. Math. Anal.}, 43(2):904--924, 2011.

\bibitem{Berman.20}
R.~J. Berman.
\newblock The {S}inkhorn algorithm, parabolic optimal transport and geometric
  {M}onge-{A}mp{\`e}re equations.
\newblock {\em Numer. Math.}, 145(4):771--836, 2020.

\bibitem{BerntonGhosalNutz.21}
E.~Bernton, P.~Ghosal, and M.~Nutz.
\newblock Entropic optimal transport: Geometry and large deviations.
\newblock {\em Preprint arXiv:2102.04397v1}, 2021.

\bibitem{Beurling.60}
A~Beurling.
\newblock An automorphism of product measures.
\newblock {\em Ann. of Math. (2)}, 72:189--200, 1960.

\bibitem{BorweinLewis.92}
J.~M. Borwein and A.~S. Lewis.
\newblock Decomposition of multivariate functions.
\newblock {\em Canad. J. Math.}, 44(3):463--482, 1992.

\bibitem{BorweinLewisNussbaum.94}
J.~M. Borwein, A.~S. Lewis, and R.~D. Nussbaum.
\newblock Entropy minimization, {$DAD$} problems, and doubly stochastic
  kernels.
\newblock {\em J. Funct. Anal.}, 123(2):264--307, 1994.

\bibitem{Carlier.21}
G.~Carlier.
\newblock On the linear convergence of the multi-marginal {S}inkhorn algorithm.
\newblock {\em Preprint hal-03176512}, 2021.

\bibitem{CarlierDuvalPeyreSchmitzer.17}
G.~Carlier, V.~Duval, G.~Peyr\'{e}, and B.~Schmitzer.
\newblock Convergence of entropic schemes for optimal transport and gradient
  flows.
\newblock {\em SIAM J. Math. Anal.}, 49(2):1385--1418, 2017.

\bibitem{CominettiSanMartin.94}
R.~Cominetti and J.~San~Mart\'{\i}n.
\newblock Asymptotic analysis of the exponential penalty trajectory in linear
  programming.
\newblock {\em Math. Programming}, 67(2, Ser. A):169--187, 1994.

\bibitem{ConfortiTamanini.19}
G.~Conforti and L.~Tamanini.
\newblock A formula for the time derivative of the entropic cost and
  applications.
\newblock {\em J. Funct. Anal.}, 280(11):108964, 2021.

\bibitem{Csiszar.75}
I.~Csisz\'{a}r.
\newblock {$I$}-divergence geometry of probability distributions and
  minimization problems.
\newblock {\em Ann. Probability}, 3:146--158, 1975.

\bibitem{Cuturi.13}
M.~Cuturi.
\newblock Sinkhorn distances: Lightspeed computation of optimal transport.
\newblock In {\em Advances in Neural Information Processing Systems 26}, pages
  2292--2300. 2013.

\bibitem{DemboZeitouni.10}
A.~Dembo and O.~Zeitouni.
\newblock {\em Large deviations techniques and applications}, volume~38 of {\em
  Stochastic Modelling and Applied Probability}.
\newblock Springer-Verlag, Berlin, 2010.
\newblock Corrected reprint of the second (1998) edition.

\bibitem{DiMarinoGerolin.20}
S.~Di~Marino and A.~Gerolin.
\newblock An optimal transport approach for the {S}chr\"{o}dinger bridge
  problem and convergence of {S}inkhorn algorithm.
\newblock {\em J. Sci. Comput.}, 85(2):Paper No. 27, 28, 2020.

\bibitem{Follmer.88}
H.~F\"{o}llmer.
\newblock Random fields and diffusion processes.
\newblock In {\em \'{E}cole d'\'{E}t\'{e} de {P}robabilit\'{e}s de
  {S}aint-{F}lour {XV}--{XVII}, 1985--87}, volume 1362 of {\em Lecture Notes in
  Math.}, pages 101--203. Springer, Berlin, 1988.

\bibitem{FollmerGantert.97}
H.~F\"{o}llmer and N.~Gantert.
\newblock Entropy minimization and {S}chr\"{o}dinger processes in infinite
  dimensions.
\newblock {\em Ann. Probab.}, 25(2):901--926, 1997.

\bibitem{GigliTamanini.21}
N.~Gigli and L.~Tamanini.
\newblock Second order differentiation formula on {${\mathrm{RCD}}^{*}(K,N)$}
  spaces.
\newblock {\em J. Eur. Math. Soc. (JEMS)}, 23(5):1727--1795, 2021.

\bibitem{Leonard.12}
C.~L\'{e}onard.
\newblock From the {S}chr\"{o}dinger problem to the {M}onge-{K}antorovich
  problem.
\newblock {\em J. Funct. Anal.}, 262(4):1879--1920, 2012.

\bibitem{Leonard.14}
C.~L\'{e}onard.
\newblock A survey of the {S}chr\"{o}dinger problem and some of its connections
  with optimal transport.
\newblock {\em Discrete Contin. Dyn. Syst.}, 34(4):1533--1574, 2014.

\bibitem{MenaWeed.19}
G.~Mena and J.~Niles-Weed.
\newblock Statistical bounds for entropic optimal transport: sample complexity
  and the central limit theorem.
\newblock In {\em Advances in Neural Information Processing Systems 32}, pages
  4541--4551. 2019.

\bibitem{Nutz.20}
M.~Nutz.
\newblock {\em Introduction to Entropic Optimal Transport}.
\newblock Lecture notes, Columbia University, 2020.
\newblock
  \url{https://www.math.columbia.edu/~mnutz/docs/EOT_lecture_notes.pdf}.

\bibitem{Pal.19}
S.~Pal.
\newblock On the difference between entropic cost and the optimal transport
  cost.
\newblock {\em Preprint arXiv:1905.12206v1}, 2019.

\bibitem{PennanenPerkkio.19}
T.~Pennanen and A.-P. Perkki\"{o}.
\newblock Convex duality in nonlinear optimal transport.
\newblock {\em J. Funct. Anal.}, 277(4):1029--1060, 2019.

\bibitem{PeyreCuturi.19}
G.~Peyr{\'e} and M.~Cuturi.
\newblock Computational optimal transport: With applications to data science.
\newblock {\em Foundations and Trends in Machine Learning}, 11(5-6):355--607,
  2019.

\bibitem{RuschendorfThomsen.93}
L.~R\"{u}schendorf and W.~Thomsen.
\newblock Note on the {S}chr\"{o}dinger equation and {$I$}-projections.
\newblock {\em Statist. Probab. Lett.}, 17(5):369--375, 1993.

\bibitem{RuschendorfThomsen.97}
L.~R\"{u}schendorf and W.~Thomsen.
\newblock Closedness of sum spaces and the generalized {S}chr\"{o}dinger
  problem.
\newblock {\em Teor. Veroyatnost. i Primenen.}, 42(3):576--590, 1997.

\bibitem{Villani.09}
C.~Villani.
\newblock {\em Optimal transport, old and new}, volume 338 of {\em Grundlehren
  der Mathematischen Wissenschaften}.
\newblock Springer-Verlag, Berlin, 2009.

\end{thebibliography}
\end{document}